\documentclass[12pt]{article}
\usepackage{amsmath}
\usepackage{amsthm}
\usepackage{amssymb}
\usepackage{graphics}
\usepackage{enumerate}   
\newcommand{\be}{\begin{equation}}
\newcommand{\ee}{\end{equation}}
\newcommand{\ben}{\begin{eqnarray*}}
	\newcommand{\een}{\end{eqnarray*}}
\newtheorem{examp}{\sc Example}
\newtheorem{remk}{\sc Remark}
\newtheorem{corol}{\sc Corollary}
\newtheorem{lemma}{\sc Lemma}
\newtheorem{theorem}{\sc Theorem}
\newtheorem{defn}{\sc Definition}
\newcommand{\bt}{\begin{theorem}}
	\newcommand{\et}{\end{theorem}}
\newcommand{\bl}{\begin{lemma}}
	\newcommand{\el}{\end{lemma}}
\newcommand{\bed}{\begin{defn}}
	\newcommand{\eed}{\end{defn}}
\newcommand{\brem}{\begin{remk}}
	\newcommand{\erem}{\end{remk}}
\newcommand{\bex}{\begin{examp}}
	
	\newcommand{\eex}{\end{examp}}
\newcommand{\bcl}{\begin{corol}}
\newcommand{\ecl}{\end{corol}}

\newcommand{\NI}{\noindent}

\newcommand{\al}{\alpha}

\newcommand{\raro}{\rightarrow}

\newcommand{\dsp}{\displaystyle}
\newcommand{\vsp}{\vskip 0.5em}

\theoremstyle{definition}
\theoremstyle{remark}

\numberwithin{equation}{section}
\numberwithin{theorem}{section}
\numberwithin{lemma}{section}

\begin{document}
\title {On Column Competent Matrices and Linear Complementarity Problem }
\author{ A. Dutta$^{a, 1}$, R. Jana$^{b, 2}$,  A. K. Das$^{b}$\\
	\emph{\small $^{a}$Jadavpur University, Kolkata , 700 032, India.}\\	
	\emph{\small $^{b}$Indian Statistical Institute, 203 B. T.
		Road, Kolkata, 700 108, India.}\\
		\emph{\small $^{1}$Email: aritradutta001@gmail.com}}

\date{}
\maketitle

\begin{abstract}
We revisit the class of column competent matrices and study some matrix theoretic properties of this class. The local $w$-uniqueness of the solutions to the linear complementarity problem can be identified by the column competent matrices. We establish some new results on $w$-uniqueness properties in connection with column competent matrices. These results are significant in the context of matrix theory as well as algorithms in operations research. We prove some results in connection with locally $w$-uniqueness property of column competent matrices. Finally we establish a connection between column competent matrices and column adequate matrices with the help of degree theory.

\NI{\bf keywords:} {Linear complementarity problem \and Column competent matrices \and $W$-uniqueness \and Column adequate matrices.}\\
\end{abstract}
\footnotetext[1]{Corresponding author}
\footnotetext[2]{The author R.Jana presently working in an integrated steel plant}

\section{Introduction}
The $w$-uniqueness property is important in the context of dynamical systems under smooth unilateral constraints. Xu \cite{xu} introduced the column competent matrices. On uniqueness, quite a large number of results are available in the literature of operations research. The study of uniqueness property of the solution is important in the context of the theory of the complementarity system as well as the method applied for finding the solution. For details see (\cite{samelson1958partition}, \cite{neogy2}, \cite{das2017finiteness},  \cite{mohan2001classes}). Ingleton \cite{ingleton1966probelm} studied the $w$-uniqueness solutions to linear complementarity problem in the context of adequate matrices. 

The linear complementarity problem can be stated as follows:
For $A\in R^{n\times n}$ and  a vector  $\,q\,\in\,R^{n},\,$ the {\it linear complementarity problem} denoted as LCP$(q,A)$ finds the solution $w\;\in R^{n}\;$ and  $z\;\in R^{n}\;$ to the following systems 
\be \label{lcp1}
\dsp {w \ - \ A z \ =\ q \ ; \ z \geq 0 \ ;w \geq 0}
\ee
\be \label{lcp2}
\dsp {w^{T}  z \ = \ 0}
\ee
or show that there does not exist any $z\;\in R^{n}\;$ and $w\;\in R^{n}\;$ satisfying the system of linear inequalities (\ref{lcp1}) and complementary condition (\ref{lcp2}).

Pang \cite{pang1988two} studied local $z$-uniqueness of solutions of a linear complementarity problem. The LCP$(q, A)$ has unique $z$-solution for all $q \in R^n$ iff $A$ is a $P$-matrix \cite{cps}. The $w$-uniqueness property is identified by a condition on $A$ related to the notion of sign-reversing. Motivated by the $w$-uniqueness results, we consider column competent matrices in the context of LCP$(q, A).$ The sufficient matrices capture many properties of positive semi definite matrices. The aim of this article is to study some matrix theoretic properties of this class and establish some new results which are useful to the solution of the LCP$(q, A).$

The paper is organised as follows. In section 2, we include few related notations and results. Section 3 presents some new results related to column competent matrices. We develop several matrix theoretic results of column competent matrices which are related to solution of linear complementarity problem. Section 4 provides a conclusion about the article.

\section{Preliminaries}

\noindent Here any vector $z\in R^{n}$ is a column vector and $z^{T}$ is the row transpose of $z.$ We write $z = z^{+} - z^{-}$ where $z^{+}_{i} = \max(0, z_i)$ and $z^{-}_{i} = \max(0, -z_i)$ for any index $i.$ If $A$ is a matrix of order $n,$ $\alpha \subseteq \{1, 2, \cdots, n\}$ and ${\beta} \subseteq \{1, 2, \cdots, n\} \setminus \alpha$ then $A_{\alpha \beta}$ is the submatrix with the rows and columns of $A$ whose indices are in $\al$ and $\beta$ respectively. A principal submatrix and a principal minor of $A$  are denoted by  $A_{\alpha \alpha}$ and $\det A_{\alpha \alpha}$ respectively. For $A \in R^{n \times n}$ and $q \in R^n,$  the feasible set  of LCP$(q, A)$ is defined by FEA$(q, A)$ $= \{z \in R^n : z \geq 0, q + Az \geq 0\}$ and the solution set is also defined by SOL$(q, A)$ $=\{z \in \text{FEA}(q, A) : z^T(q + Az) = 0\}.$  A $z$-solution, $\tilde{z}$ is called locally unique if $\exists$ a neighborhood of $\tilde{z}$ within which $\tilde{z}$ is the only $z$-solution.  A $w$-solution, $\tilde{w}=A\tilde{z}+q$, is called locally unique if $\exists$ a neighborhood of $\tilde{w}$ in which $\tilde{w}$ is the only $w$-solution.  Let $\psi:R^n \to R^n$ and the kernel of the function $\psi$ is defined by ker \ $\psi=\{z\in R^n: \psi(z)=0\}.$ The kernel of a matrix $A \in R^{n \times n}$ is defined by ker\ $A$$=\{z\in R^n:Az=0\}.$ Now we define the column competent matrix.
\begin{defn}\cite{xu}
	The matrix $A$ is said to be
	column competent if
	$z_i(Az)_i=0 ,\  \ i =1,2,\cdots,n  \implies Az=0.$
	
\end{defn}
Column competent matrices can be singular or nonsingular matrices. Note that all singular matrices need not be column competent matrices. Consider $A=$ $\left[\begin{array}{rr}
1 & 0 \\
1 & 0  \\ 
\end{array}\right]$ which is a singular matrix. For any  $z=\left[\begin{array}{rr} 
0  \\
k   \\
\end{array}\right], \ k \in R, \   z_i(Az)_i=0, \ i =1,2 $ implies that $Az=0.$ Consider another $A=$ $\left[\begin{array}{rr}
1 & 1 \\
0 & 0  \\ 
\end{array}\right].$ It is easy to show that $z_i(Az)_i=0,  \ i =1,2 $ does not imply $Az=0.$ Hence $A$ is not a column competent matrix. Let  $A=$ $\left[\begin{array}{rrr} 
1 & 4 & 3 \\
2 & 1 & 5  \\
3 & 2 & 0\\
\end{array}\right],$ for $z= \left[\begin{array}{rrr} 
0  \\
0   \\
1 \\
\end{array}\right]$ $z_i(Az)_i=0, \  \ i=1,2,3$ but $Az=\left[\begin{array}{rrr} 
3  \\
5   \\
0 \\
\end{array}\right]\neq 0.$ Here $A$ is a nonsingular matrix but not a  column competent matrix.\\

Now we define $\psi:R^n \to R^n$ where $\psi(z)=z*(Az)$ and $z*(Az)$ is the Hadamard product defined by $(z*(Az))_i=z_i * (Az)_i , \ \forall \ i.$ Note that the product is associative, distributive and commutative.
\begin{defn}\cite{xu}
	In view of Hadamard product, a matrix $A$ is said to be
	column competent if $\text{ker}\ \psi=\text{ker\ }A.$
\end{defn}
Column adequate matrices are related to column competent matrices. We start with definition of column adequate matrices.
\begin{defn}\cite{cps}
	The matrix $A$ is said to be
	column adequate if
	$z_i(Az)_i \leq 0 , \ i =1,2,\cdots,n  \implies Az = 0.$
\end{defn} 
We state the following lemma and theorems which are useful for the subsequent sections.
\begin{lemma}\cite{xu} \label{lem1}
	The matrix $A$ is said to be non-degenerate if and only if $\text{ker} \ \psi = \{0\}.$
\end{lemma}
\begin{theorem} \label{thm1} \cite{xu}
	The following statements are equivalent.
	\begin{enumerate}[(i)]
		\item $A$ is column competent.
		\item For all vector $q,$ the LCP$(q, A)$ has a finite number (possibly zero) of $w$-solutions.
		\item For all vector $q,$ any $w$-solution of the LCP$(q, A),$ if it exists, must be locally $w$-unique.
	\end{enumerate}
\end{theorem}

\begin{theorem} \label{thm2} \cite{xu}
	The following statements are equivalent.
	\begin{enumerate}[(i)]
		\item 
		(a) $A$ is column competent.\\
		(b) $A$ is a $P_0$-matrix.\\
		\item $A$ is column adequate.
	\end{enumerate}
\end{theorem}

\begin{theorem} \label{thm3} \cite{cps}
	Let $A \in R^{n \times n}$ be a $E_0$-matrix. Then the following statements are equivalent.
	\begin{enumerate}[(i)]
		\item $A \in R_0.$
		\item $A \in R.$
	\end{enumerate}
\end{theorem}

We say that $A\in R^{n\times n}$ is a 

\NI $-$ {\it $Q$}-matrix if for every $q\in R^{n},$ LCP$(q,A)$ has a solution. \\
\NI $-$ {\it $Q_{0}$}-matrix if for any $q\in R^{n},$ feasibility implies solvability. \\
\NI $-$ {\it $P$}-matrix if for each vector $z \neq 0$ there exists an index $i$ such that $\max_{(z_i \neq 0)}z_i(Az)_i > 0.$\\
\NI $-$ {\it $P_0$}-matrix if for each vector $z \neq 0$ there exists an index $i$ such that $\max_{(z_i \neq 0)}z_i(Az)_i \geq 0.$\\
\NI $-$ {\it $R_0$}-matrix if LCP$(0,A)$ has unique solution.\\
\NI $-$ {\it principally non-degenerate} if it has no principal submatrix which has determinant zero.\\For further details about the matrix classes in linear complementarity problem see 
(\cite{das11},\cite{mohan2001more}, \cite{neogy2006some}, \cite{jana2019hidden}, \cite{jana2019more}).

The principal pivot transform (PPT) has an important role in the study of matrix classes and linear complementarity problem. The {\it principal pivot transform} of $A \in R^{n \times n}$ with real entries, with respect to $\al\subseteq
\{1,2,\ldots,n\}$ is defined as the matrix given by
$$ \dsp {
	A' = \left[ \begin{array}{cc}
	A^{'}_{\al\al} & A^{'}_{\al\bar{\al}}\\
	A^{'}_{\bar{\al}\al} & A^{'}_{\bar{\al}\bar{\al}}
	\end{array} \right]
} $$
{where}
$A^{'}_{\al\al} =(A_{\al\al})^{-1},\;
A^{'}_{\al\bar{\al}}$=$-(A_{\al\al})^{-1} 
A_{\al\bar{\al}},\;\,$$A^{'}_{\bar{\al}\al}=
A_{\bar{\al}\al}(A_{\al\al})^{-1}$ and $A^{'}_{\bar{\al}\bar{\al}}=
A_{\bar{\al}\bar{\al}}-A_{\bar{\al}\al}(A_{\al\al})^{-1}
A_{\al\bar{\al}}.$ 

Here PPT is only identified with respect to those $\al$ for which
$\det A_{\al\al} \neq 0.$  When $\al=\emptyset$, by convention $\det A_{\al\al}=1$ and $A'=A.$ Here $A^{'}_{\bar{\al}\bar{\al}}$ is said to be Schur complement of $A.$ We denote the PPT of A as $A^{'} = \mathcal{P}_{\alpha}(A).$  The schur complement of $A_{\alpha \alpha}$ in $A=\left[ \begin{array}{cc}
A_{\al\al} & A_{\al\bar{\al}}\\
A_{\bar{\al}\al} & A_{\bar{\al}\bar{\al}}
\end{array} \right]$ is a principal submatrix of the principal pivot transform $A'$. For details of PPT see (\cite{neogy2}, \cite{neogy1}, \cite{das}).

We establish a connection between competent matrices and adequate matrices using degree theoretic approach. We provide a brief details about degree theory in the subsequent section. 
\subsection{Degree theory}
Let $f_{A}: R^n \raro R^n$ be a piecewise linear mapping for a given matrix $A \in R^{n \times n}$ defined as $f_{A}(e_i) = e_i$ and $f_{A}(-e_i) = -Ae_i \ \forall \ i.$ We write for any $z \in R^n,$ 
$$f_{A}(z) = z^{+} - Az^{-}.$$ For details see \cite{mohan2001classes}. It is clear that LCP$(q, A)$ is equivalent to find a vector $z \in R^{n}$ such that $f_{A}(z) = q.$ If $z$ belongs to the interior of some orthants of $R^n$ and $\det A_{\al \al} \neq 0$ where $\al = \{i : z_i < 0\},$ then the index of $f_{A}(z)$ at $z$ is well defined and can be written as 
$$\text{ind} f_{A}(q, z) = sgn(\det A_{\al \al}).$$ Note that the cardinality of $f_{A}^{-1}(q)  $ denotes the number of solutions of LCP$(q, A).$ Particularly, if $q$ is non-degenerate with respect
to $A,$ each index of $f_{A}$ is well defined and we can define local degree of $A$ at $q.$ It can be denoted as deg$_{A}(q).$ For details see (\cite{cps}, chapter 6). We state the following theorem from \cite{mohan2001classes}, which will be required to prove one of our result.

\begin{theorem} \label{thm4}
	Let $A \in$$ R^{n \times n}.$ Let $K(A)$ denote the union of all the facets of the complementary cones of $(I,{}-A)$. 
	Consider $q \in R^n \setminus k(A)$ where $q$ is non-degenerate with respect to $A.$ Let $\beta \subseteq \{1, 2, \cdots, n\}$ be such that $\det A_{\beta \beta}\neq 0.$ Suppose $A'$ is a PPT of $A$ with respect to $\beta.$ Then
	$\text{deg}_{A'}(q') = sgn( \det A_{\beta \beta}) \cdot \text{deg}_{A}(q).$
\end{theorem}

\section{Results on Column Competent Matrices}
Hadamard product is important to characterize the complementary condition. Here we show that the property of column competent matrix is related to Hadamard product.
\begin{theorem} \label{1}
	Suppose $A$ is a column competent matrix and the function $\psi:R^n \to R^n$ defined by $\psi(z)=z*(Az)$ where $z*(Az)$ is the Hadamard product. Then $\text{ker} \ \psi= \text{ker} A. $
\end{theorem}
\begin{proof}
	Let $A$ be a column competent matrix. Then for a vector $z \in R^n, z_i(Az)_i=0, \ i=1,2,\cdots,n \implies Az = 0.$ Hence $z \in\text{ker} \ \psi$ implies $z \in \text{ker} A.$ So we write $\text{ker} \ \psi \subseteq \text{ker} A.$ Again by definition $\text{ker} A \subseteq \text{ker} \ \psi.$ Therefore $\text{ker} \ \psi= \text{ker} A.$
\end{proof}
The following result provides a characterization of non-degenerate column competent matrices. 
\begin{theorem}
	Let $A \in R^{n\times n}$ be a non-degenerate column competent matrix. Then $A \in R_0.$
\end{theorem}
\begin{proof}
	Let $A$ be a non-degenerate column competent matrix. By Lemma \ref{lem1}, $\text{ker} \ \psi = \{0\}$ where $\psi(z)=z*Az.$ By Theorem \ref{1}, we can write $\text{ker}\ \psi=\text{ker\ }A= \{0\}.$ Let $z$ be the solution of LCP$(0,A).$ Then $z_i(Az)_i=0, \ i=1,2, \cdots,n.$ This implies that $Az=0.$ Hence $z=0.$ Therefore, LCP$(0, A)$ has only one solution zero. Hence  $A$ is a $R_0$- matrix.
\end{proof}
Note that column competent matrix need not be a $P_0$- matrix in general. Consider the matrix $A=$ $\left[\begin{array}{rr} 
2 & 1 \\
1 & -1 \\
\end{array}\right].$ We show that $A$ is a column competent matrix but not a $P_0$-matrix. Now we establish the following result.
\begin{theorem}\label{33}
	Suppose $A$ is a column competent matrix with $A \in P_0.$ Then for $0 \neq z\geq 0,$ $(z,0)$ is the solution of LCP$(0,A).$	
\end{theorem}
\begin{proof}
	Let $A\in R^{n\times n}$ be a column competent matrix with $A \in P_0.$ Then for each $0 \neq z,$ $\max_{i} z_i(Az)_i \geq 0, z_i\neq 0.$  If $z_i(Az)_i=0, \ i=1,2,\cdots,n$ implies that $Az=0.$ Then $(z,0), z\geq 0$ is the solution of LCP$(0,A).$
\end{proof}
Now we consider the matrix $A=$ $\left[\begin{array}{rr} 
2 & -1 \\
-4 & 2 \\
\end{array}\right]$ is column competent as well as $P_0$ and $(\left[\begin{array}{rr} 
1\\
2 \\
\end{array}\right], \left[\begin{array}{rr} 
0\\
0 \\
\end{array}\right])$ is a solution of LCP$(0,A).$  Note that this can be explained using the Theorem \ \ref{33}.
\begin{theorem}
	Let $A$ be a column competent matrix. Suppose $z \geq 0$ and $z_i(Az)_i=0, \  i= 1,2, \cdots, n.$ Then LCP$(0, A)$ has the solution $(z, 0).$ 
\end{theorem}
\begin{proof}
	Since $A$ is a column competent matrix, then for $z \geq 0$ and $z_i(Az)_i=0,  \  i = 1, 2, \cdots, n.$ This implies that $Az = 0.$  Therefore $(z,0)$  is the solution of LCP$(0, A).$ 
\end{proof}	

Xu \cite{xu} showed that if $A$ is a column competent matrix then $DAD^T$ is a column competent matrix where $D$ is a diagonal matrix. In the next theorem, we prove that column competent matrices with some additional assumptions are invariant under principal rearrangement. For any principal submatrix $A_{\alpha \alpha}$ of $A,$it is possible to rearrange principally the rows and columns of $A$ in such a way that $A_{\alpha \alpha}$ becomes a leading principal submatrix in the rearranged matrix $PAP^T.$
\begin{theorem}
	Suppose $A$ is a column competent matrix. If for any $z \in R^n,$ either $z_i(Az)_i \geq 0$ or $z_i(Az)_i \leq 0$ for all $i,$ then $PAP^T$ is also column competent where $P$ is a permutation matrix.
\end{theorem}
\begin{proof}
	Let for any $z \in R^n,$ $y = Pz.$ Consider $y_i(PAP^Ty)_i=0$   for all $ i.$ This implies that $(Pz)_i(PAP^TPz)_i=0$ for all $i.$ We know that \begin{center}
		$z^TP^TPAP^TPz=$ $\sum_{i=1}^{n}(Pz)_i(PAP^TPz)_i$ $= 0.$
	\end{center}  Hence $z^TAz = 0$ as $P^TP = I.$ We write $\sum_{i=1}^{n}z_i(Az)_i=0.$ It means $z_i(Az)_i=0, \ i=1,2,\cdots,n. $ As $A$ is a column competent matrix, $Az=0.$ Therefore $AP^TPz=0.$  Hence $(PAP^T)(Pz)=0.$ Hence $PAP^T $ is column competent.
\end{proof}
\begin{theorem}
	Let $A$ be a $Z$- matrix. Suppose $z_i(Az)_i=0$ for all $ i, $ $A|z| \geq 0 $ and $Az \leq 0.$ Then  $A$ is a column competent matrix.  
\end{theorem}
\begin{proof}
	Suppose $A$ is a $Z$- matrix. Consider $z_i(Az)_i=0$ for all $ i, $ $A|z| \geq 0$ and $Az \leq 0.$  As $A$ is a $Z$- matrix, $Az \geq A|z| \geq 0.$ Now this implies that $Az = 0.$ Therefore $A$ is a column competent matrix.
\end{proof}

Consider $A=$ $\left[\begin{array}{rrr} 
1 & 1 & 4\\
2 & 2 & 5 \\
3 & 4 & 1\\
\end{array}\right].$ Note that $A$ is a $R_0$-matrix. Now for  $z=\left[\begin{array}{rrr} 
1\\
-1 \\
0\\
\end{array}\right],$ $z_i(Az)_i=0, \ i=1,2,3$ but $Az\neq 0.$ Hence $A$ is not a column competent matrix. The class of non-degenerate matrices play an important role to characterize certain uniqueness properties of the solutions of LCP$(q, A).$ We prove the following theorem to establish the relation between principally non-degenerate matrices and column competent matrices.
\begin{theorem}
	Let $A$ be a principally nondegenerate matrix. Then $A$ is column competent.
\end{theorem}
\begin{proof}
	Let $A$ be a principally non-degenerate matrix. Assume that $A$ is not a column competent matrix. Hence $\exists$ a $0 \neq z \in R^n$ such that $z_i(Az)_i=0, \ i = 1, 2, \cdots, n$ but $Az \neq 0.$  Without loss of generality, consider $z = \left[\begin{array}{rr} 
	z_{\alpha} \\
	z_{\beta} \\
	\end{array}\right] \neq 0$ where $ z_{\alpha} \neq 0, z_{\beta} = 0$ and $A=\left[\begin{array}{rr} 
	A_{\alpha \alpha} & A_{\alpha \beta} \\
	A_{\beta \alpha} & A_{\beta \beta} \\
	\end{array}\right].$ Then we consider the following cases:\\

{case1:}
	 Let $\alpha=\{1,2,\cdots,n\}$ and $\beta =\emptyset .$ Then $z=z_\alpha$ and $z_i(Az)_i=0,  \ i \in \alpha$. It implies $Az=0,$ contradicts the fact that $Az \neq 0.$\\

       {case2: }  Let $\alpha \subset \{1,2,\cdots,n\}$ and $\beta= \{1,2,\cdots,n\} \setminus \alpha.$ Consider $(z_{\alpha})_i(A_{\alpha \alpha}z_{\alpha})_i=0, \ i\in \alpha.$ This implies $A_{\alpha \alpha}z_{\alpha}=0.$ As $z_{\alpha} \neq 0,$ $A_{\alpha \alpha}$ is a singular matrix. It contradicts that the matrix $A$ is a principally non-degenerate matrix.\\

	Therefore $A$ is a column competent matrix.
\end{proof}

Here we consider $A=$ $\left[\begin{array}{rrr} 
3 & -2 & 0\\
-2 & 1 & 1 \\
-3 & 2 & 0\\
\end{array}\right].$  For $z=\left[\begin{array}{rrr} 
2k\\
3k \\
k\\
\end{array}\right], k \in R, \ z_i(Az)_i=0, \ i=1,2,3 $ implies that $Az=0.$ Hence  $A$ is a column competent matrix. However $A$ is neither an adequate matrix nor a sufficient matrix. For  details of sufficient matrices see (\cite{valiaho1996p}, \cite{sun2006smoothing}, \cite{den1993linear}).
\vsp
Now we develop a necessary and sufficient condition for column competent matrices.

\begin{theorem}
	Let $A \in R^{n \times n}$. The following two statements are equivalent: 
	\begin{enumerate}[(a)]
		\item $A$ is column competent.
		\item For $0 \neq z=$	$\left[\begin{array}{rr} 
		z_{\alpha} \\
		z_{\beta} \\
		\end{array}\right] \geq 0$ with $z_{\beta}=0 $ and the submatrix $A_{\alpha \alpha}$ is singular with $A_{\alpha \alpha}z_{\alpha}=0 $ where $\alpha \cup \beta =\{1,2,\cdots,n\}$ and $\alpha \cap \beta= \emptyset,$ the system 	\begin{equation}\label{2}
		\left[\begin{array}{rr} 
		A_{\alpha \alpha} & A_{\alpha \beta} \\
		A_{\beta \alpha} & A_{\beta \beta} \\
		\end{array}\right]\left[\begin{array}{rr} 
		z_{\alpha} \\
		z_{\beta} \\
		\end{array}\right] \neq 0
		\end{equation} 
		has no solution.
	\end{enumerate}
\end{theorem}
\begin{proof}
	$(a) \implies (b).$ Suppose $A$ is column competent and Equation \ref{2} is consistent. Let $\left[\begin{array}{rr} 
	z_{\alpha} \\
	z_{\beta} \\
	\end{array}\right]$ satisfies the Equation \ref{2} where  $z_{\beta}=0 $ and the submatrix $A_{\alpha \alpha}$ is singular with $A_{\alpha \alpha}z_{\alpha}=0 $ where $\alpha \cup \beta =\{1,2,\cdots,n\}$ and $\alpha \cap \beta= \emptyset$.  Here $(z_{\alpha})_i(A_{\alpha \alpha}z_{\alpha})_i=0$ and $(z_{\beta})_i(A_{\beta \alpha}z_{\alpha})_i=0.$ But $\left[\begin{array}{rr} 
	z_{\alpha} \\
	z_{\beta} \\
	\end{array}\right]$ satisfies  \ref{2} which contradicts that $A$ is column competent. 
	
	$(b) \implies (a).$ Conversely, let $x \in R^n$ be a vector such that $x_i(Ax)_i=0$ for all $i.$ Consider $A$ is not a column competent matrix. Suppose $x_\alpha = z_\alpha$ and $x_\beta = z_\beta$ with $0 \neq z=$ $\left[\begin{array}{rr} 
	z_{\alpha} \\
	z_{\beta} \\
	\end{array}\right] \geq 0,$ $z_{\beta}=0$ and the submatrix $A_{\alpha \alpha}$ is singular with $A_{\alpha \alpha}z_{\alpha}=0$ where $\alpha \cup \beta =\{1,2,\cdots,n\}$ and $\alpha \cap \beta= \emptyset.$ But the system \ref{2} has no solution, i.e. $x$ does not satisfy Equation \ref{2}. Therefore $A$ is column competent.
\end{proof}
Now we prove the following sufficient condition related to the PPT of column competent matrices.
\begin{theorem}
	Let $A_{\alpha \alpha}$ and the Schur complement $A/A_{\alpha \alpha}$ be nonsingular  of the square matrix $A=\left[\begin{array}{rr} 
	A_{\alpha \alpha} & A_{\alpha \beta} \\
	A_{\beta \alpha} & A_{\beta \beta} \\
	\end{array}\right]$ where $\alpha \cup \beta =\{1,2,\cdots,n\}$ and $\alpha \cap \beta= \emptyset.$ If $A$ is column competent, then $A'=\mathcal{P}_{\alpha}(A)$ is column competent.
\end{theorem}
\begin{proof}
	Let $w=A'z$ and $z*w=0$ where $*$ is the Hadamard product. Thus we write
	\begin{equation}\label{system33}
	\left[\begin{array}{rr} 
	w_{\alpha} \\
	w_{\beta} \\
	\end{array}\right] = \left[\begin{array}{rr} 
	A'_{\alpha \alpha} & A'_{\alpha \beta} \\
	A'_{\beta \alpha} & A'_{\beta \beta} \\
	\end{array}\right] \left[\begin{array}{rr} 
	z_{\alpha} \\
	z_{\beta} \\
	\end{array}\right].
	\end{equation} 
	The condition $z*w=0$ means $\left[\begin{array}{rr} 
	z_{\alpha} \\
	z_{\beta} \\
	\end{array}\right] * \left[\begin{array}{rr} 
	w_{\alpha} \\
	w_{\beta} \\
	\end{array}\right] = \left[\begin{array}{rr} 
	w_{\alpha}* z_{\alpha} \\
	w_{\beta}* z_{\beta} \\
	\end{array}\right]=0.$ Since $A'=\mathcal{P}_{\alpha}(A),$ we have
	\begin{equation}
	\left[\begin{array}{rr} 
	z_{\alpha} \\
	w_{\beta} \\
	\end{array}\right]=\left[\begin{array}{rr} 
	A_{\alpha \alpha} & A_{\alpha \beta} \\
	A_{\beta \alpha} & A_{\beta \beta} \\
	\end{array}\right]\left[\begin{array}{rr} 
	w_{\alpha} \\
	z_{\beta} \\
	\end{array}\right].
	\end{equation} The matrix $A$ is column competent implies that $\left[\begin{array}{rr} 
	A_{\alpha \alpha} & A_{\alpha \beta} \\
	A_{\beta \alpha} & A_{\beta \beta} \\
	\end{array}\right]\left[\begin{array}{rr} 
	w_{\alpha} \\
	z_{\beta} \\
	\end{array}\right]=0.$ It follows that $\left[\begin{array}{rr} 
	z_{\alpha} \\
	w_{\beta} \\
	\end{array}\right]=0. $ From \ref{system33}, we get $A'_{\beta \alpha}z_{\alpha} + A'_{\beta \beta}z_{\beta} = 0.$ Hence $A'_{\beta \beta}z_{\beta}=0$ implies that $z_{\beta}=0$ as $A'_{\beta \beta}=A/A_{\alpha \alpha}$ is nonsingular. Clearly, $w_{\alpha}=0.$ Hence $\left[\begin{array}{rr} 
	w_{\alpha} \\
	w_{\beta} 
	\end{array}\right] = \left[\begin{array}{rr} 
	A'_{\alpha \alpha} & A'_{\alpha \beta} \\
	A'_{\beta \alpha} & A'_{\beta \beta} \\
	\end{array}\right] \left[\begin{array}{rr} 
	z_{\alpha} \\
	z_{\beta} \\
	\end{array}\right]=0.$ Therefore $A'$ is column competent.
\end{proof}
\begin{theorem}
	Let $A$ be a column competent matrix where $A_{\alpha \alpha}$ and the Schur complement $A/A_{\alpha \alpha}$ be nonsingular  of the square matrix $A=\left[\begin{array}{rr} 
	A_{\alpha \alpha} & A_{\alpha \beta} \\
	A_{\beta \alpha} & A_{\beta \beta} \\
	\end{array}\right].$ If $A \in E_0 \cap R_0, $ then $A$ is column adequate.
\end{theorem}
\begin{proof}
	Suppose $A$ is not a column adequate matrix but is column competent. By Theorem \ref{thm2}, $A$ is not a $P_0$- matrix. Then there exists $\beta \subseteq \{1,2,\cdots,n\}$ such that $\det A_{\beta \beta}<0.$ Let $A \in E_0 \cap R_0.$ It follows from the Theorem \ref{thm3} that $A \in R.$ Then deg$_{A}(q) = 1$ for any $q.$ Let $A'$ be a principal pivot transform of $A.$ Then $A' \in R.$ Hence deg$_{A'}(q') = 1.$ By Theorem \ref{thm4},  deg$_{A'}(q')=$ deg$_{A}(q). sgn(\det A_{\beta \beta}).$ It implies that deg$_{A'}(q')=-1.$ This contradicts that $A$ is not a $P_0$-matrix. Therefore $A$ is column adequate matrix. 
\end{proof}

\subsection{Solution of Linear Complementarity Problem with Column Competent Matrices}
We begin with some examples of $w$-uniqueness of the solution. Consider the column competent matrix $A=\left[\begin{array}{rr} 
-1 & 3 \\
2 & -6 \\
\end{array}\right], q=\left[\begin{array}{rr} 
1\\
-2 \\
\end{array}\right].$ This LCP$(q,A)$ has solution $z=\left[\begin{array}{rr} 
4\\
1 \\
\end{array}\right] $ and $w=\left[\begin{array}{rr} 
0\\
0 \\
\end{array}\right].$ In the neighbourhood of $z$ there is another solution $z'=\left[\begin{array}{rr} 
4.0100\\
1.0033 \\
\end{array}\right] $ and $w'=w=\left[\begin{array}{rr} 
0\\
0 \\
\end{array}\right].$ \\
We consider another  matrix  $A=\left[\begin{array}{rrr} 
-2 & 1 & 3 \\
4 & -2 & -6 \\
1 & -1 & -1\\
\end{array}\right], q=\left[\begin{array}{rrr} 
1\\
-2 \\
1\\
\end{array}\right].$ For $z=\left[\begin{array}{rrr} 
2k\\
k\\
k\\
\end{array}\right], $\ $k \in R, z_i(Az)_i=0, \ i=1,2,\cdots,n$ implies that $Az=0.$ So $A$ is a column competent matrix. This LCP$(q,A)$ has solution $z=\left[\begin{array}{rrr} 
4\\
4\\
1 \\
\end{array}\right] $ and $w=\left[\begin{array}{rr} 
0\\
0 \\
0\\
\end{array}\right].$ In the neighbourhood of $z$ there is another solution $z'=\left[\begin{array}{rr} 
4.02\\
4.01\\
1.01 \\
\end{array}\right] $ and $w'=w=\left[\begin{array}{rr} 
0\\
0 \\
0\\
\end{array}\right].$ \\

Now we prove the following two results in connection with locally $w$-uniqueness property of the column competent matrices. The following two results state the necessary and sufficient condition that $A$ is a column competent matrix in the system of linear complementarity problem. 
\begin{theorem}
	Suppose $(w^*, z^*)$ is the solution of LCP$(q, A)$ such that $w^*= q + Az^*.$ Let $\alpha=\{i:{w_i}^*>0\},$  $\beta=\{i:{w_i}^*=0\}$ be the index set. Further consider that the submatrix $A_{\alpha \alpha}$ is nonsingular. If $A=\left[\begin{array}{rr} 
	A_{\alpha \alpha} & A_{\alpha \beta} \\
	A_{\beta \alpha} & A_{\beta \beta} \\
	\end{array}\right]$ is a column competent matrix, then $(w_\alpha, z_{\beta}) = (0, 0)$ is the only solution of the system:
	\begin{equation} \label{system44}
	\begin{split}
	z_{\alpha}=A'_{\alpha \alpha}w_{\alpha}+A'_{\alpha \beta}z_{\beta}=0 \\
	w_{\beta}=A'_{\beta \alpha}w_{\alpha}+A'_{\beta \beta}z_{\beta}=0 \\
	w_{\alpha} > 0 \\
	z_{\beta} > 0,
	\end{split}
	\end{equation}\\ where $A'_{\alpha \alpha}=(A_{\alpha \alpha})^{-1},$ $A'_{\alpha \beta}=-(A_{\alpha \alpha})^{-1}A_{\alpha \beta},$ $A'_{\beta \alpha}=A_{\beta \alpha}(A_{\alpha \alpha})^{-1}$ and $A'_{\beta \beta}= A_{\beta \beta}-A_{\beta \alpha}(A_{\alpha \alpha})^{-1}A_{\alpha \beta}.$
\end{theorem}
\begin{proof}
	Let $A$ be a column competent matrix. Then by Theorem \ref{thm1}, it is locally $w$-unique. Suppose $w^*$ is locally unique solution of LCP$(q, A)$ such that $w^* = q + Az^*$ and the system (\ref{system44}) has a nonzero solution $(\bar{w}_\alpha, \bar{z}_{\beta}).$\\ Now
	$	\left[\begin{array}{rr} 
	\bar{z}_{\alpha} \\
	\bar{w}_{\beta} \\
	\end{array}\right] = \left[\begin{array}{rr} 
	A'_{\alpha \alpha} & A'_{\alpha \beta} \\
	A'_{\beta \alpha} & A'_{\beta \beta} \\
	\end{array}\right] \left[\begin{array}{rr} 
	\bar{w}_{\alpha} \\
	\bar{z}_{\beta} \\
	\end{array}\right]=0$
	implies that $	\left[\begin{array}{rr} 
	\bar{w}_{\alpha} \\
	\bar{w}_{\beta} \\
	\end{array}\right] = \left[\begin{array}{rr} 
	A_{\alpha \alpha} & A_{\alpha \beta} \\
	A_{\beta \alpha} & A_{\beta \beta} \\
	\end{array}\right] \left[\begin{array}{rr} 
	\bar{z}_{\alpha} \\
	\bar{z}_{\beta} \\
	\end{array}\right].$
	Clearly, $\bar{w}=A \bar{z}$ and $(w^{*})^T \bar{z}=0,$ $(\bar{w})^T z^*=0. $ Hence $(w^{*} +\lambda\bar{w}, z^{*} +\lambda\bar{z})$ solves LCP$(q, A)$ for all $\lambda \geq 0.$ This contradicts the local uniqueness of $w^*.$ Therefore, $(w_\alpha, z_{\beta}) = (0,0)$ is the only solution of the system (\ref{system44}).
\end{proof}
\begin{theorem}
	Suppose $(w^*, z^*)$ is the solution of LCP$(q, A)$ such that $w^* = q + Az^*$ where $\alpha=\{i:{w_i}^*>0\}$ and $\beta=\{i:{w_i}^*=0\}.$ Further suppose 	$	\left[\begin{array}{rr} 
	{z}_{\alpha} \\
	{w}_{\beta} \\
	\end{array}\right] = \left[\begin{array}{rr} 
	A'_{\alpha \alpha} & A'_{\alpha \beta} \\
	A'_{\beta \alpha} & A'_{\beta \beta} \\
	\end{array}\right] \left[\begin{array}{rr} 
	{w}_{\alpha} \\
	{z}_{\beta} \\
	\end{array}\right]=0, $ $w_{\alpha} > 0,
	z_{\beta} > 0.$\  If
	$(z_\alpha, z_{\beta}) = (0, 0)$ is the only solution of 
	$w_{\beta}=A_{\beta \alpha}z_{\alpha} + A_{\beta \beta}z_{\beta} =0$
	then $A=\left[\begin{array}{rr} 
	A_{\alpha \alpha} & A_{\alpha \beta} \\
	A_{\beta \alpha} & A_{\beta \beta}
	\end{array}\right]$ is column competent.
\end{theorem}
\begin{proof}
	Suppose the matrix $A$ is not column competent. So $w^*$ is not locally unique.  Now
	$	\left[\begin{array}{rr} 
	{z}_{\alpha} \\
	{w}_{\beta} \\
	\end{array}\right] = \left[\begin{array}{rr} 
	A'_{\alpha \alpha} & A'_{\alpha \beta} \\
	A'_{\beta \alpha} & A'_{\beta \beta} \\
	\end{array}\right] \left[\begin{array}{rr} 
	{w}_{\alpha} \\
	{z}_{\beta} \\
	\end{array}\right]=0$ implies that 	$	\left[\begin{array}{rr} 
	{w}_{\alpha} \\
	{w}_{\beta} \\
	\end{array}\right] = \left[\begin{array}{rr} 
	A_{\alpha \alpha} & A_{\alpha \beta} \\
	A_{\beta \alpha} & A_{\beta \beta} \\
	\end{array}\right] \left[\begin{array}{rr} 
	{z}_{\alpha} \\
	{z}_{\beta} \\
	\end{array}\right]$ and $(w^{*})^T {z}=0,$ $({w})^T z^*=0. $ Hence $(w^{*} +\lambda{w}, z^{*} +\lambda{z})$ solves LCP$(q, A)$ for all $\lambda \geq 0.$  Hence $\exists$ a sequence of vectors $\{\bar{w}^k\}$ converging to $w^*$ such that each $(\bar{w}^k,\bar{z}^k)=(w^{*} +\lambda^k{w},z^{*} +\lambda^k{z})$ is a solution of LCP$(q, A)$ with $\bar{w}^k = q + A\bar{z}^k.$ Since $\bar{w}^k \to w^*$ and $\bar{z}^k \to z^*,$ it follows that $\bar{w}_{\alpha}^k >0, \bar{z}_{\beta}^k >0.$ By complementarity $\bar{z}_{\alpha}^k =0, \bar{w}_{\beta}^k =0.$ Consider $v^k=\bar{w}^k-w^*$ and $u^k=\bar{z}^k-z^*.$ The normalized sequence $\{v^k/ \|{v^k}\|\}$ is bounded and converges to $v^* \neq 0$ as $k \to \infty.$ Similarly, the normalized sequence $\{u^k/ \|{u^k}\|\}$ is bounded and converges to $u^* \neq 0$ as $k \to \infty.$ Now for all large $k,$ we have $\bar{w}_{\beta}^k - w_{\beta}^*=\lambda^kw_{\beta}=0=A_{\beta \alpha}u_{\alpha}^k + A_{\beta \beta}u_{\beta}^k.$ Thus dividing by $\|u^k\|$ and $k \to \infty,$ we have $A_{\beta \alpha}{u_{\alpha}}^*+A_{\beta \beta}{u_{\beta}}^*=0.$ Therefore, $u^*=\left[\begin{array}{rr} 
	{u_{\alpha}}^*\\
	{u_{\beta}}^* \\
	\end{array}\right] \neq 0$ is the nonzero solution of system $w_{\beta}=A_{\beta \alpha}z_{\alpha} + A_{\beta \beta}z_{\beta} =0.$ It contradicts that $(z_\alpha, z_{\beta})=(0,0)$ is the only solution of the system $w_{\beta}=A_{\beta \alpha}z_{\alpha} + A_{\beta \beta}z_{\beta} =0.$ Hence $A$ is column competent.
\end{proof}
\section{Conclusion}
The complementary condition is an important issue in operations research. The concept of matrix theoretic approach helps to develop many theory of linear complementary problem. In this study we consider column competent matrix in the context of local $w$-uniqueness property which is important both for the theory as well as solution method of complementarity problrm. The results based on w-uniqueness and column competent matrix class motivate future study and application in matrix theory.
\vsp
\section{Acknowledgement}
The author A. Dutta is thankful to the Department of Science and Technology, Govt. of India, INSPIRE Fellowship Scheme for financial support.

\vsp
\bibliographystyle{plain}
\bibliography{bibfile}
\end{document}